\documentclass[12pt]{article}
\usepackage[utf8]{inputenc}
\usepackage[margin= 3 cm, bottom=20mm,footskip=10mm,top=20mm]{geometry}
\usepackage{hyperref}
\usepackage{amsthm}
\usepackage{amsmath}
\usepackage{amssymb}
\usepackage{amsfonts}
\usepackage[noadjust]{cite}
\usepackage{epsfig}
\usepackage{mathtools}
\usepackage{enumitem}
\usepackage{stackengine,graphicx}
\usepackage{tikz}
\usepackage{comment}

\newtheorem{thm}{Theorem}[section]
\newtheorem{lem}[thm]{Lemma}

\newtheorem{qu}[thm]{Question}

\newtheorem*{unnumberedclm}{Claim}
\newtheorem{clm}{Claim}

\newtheorem{remark}[thm]{Remark}

\newtheorem{defi}[thm]{Definition}
\theoremstyle{definition}

\theoremstyle{plain}

\newcommand{\D}{\mathsf{d}}

\newcommand{\NN}{\mathbb{N}}

\newcommand{\ZZ}{\mathbb{Z}}

\newcommand{\eps}{\varepsilon}

\newcommand{\dcut}{\textrm{d}_{\square}}
\newcommand{\dlp}{\textrm{dist}_{LP}}

\begin{document}

\title{On pattern-avoiding permutons
\thanks{
JH: Research supported by Czech Science Foundation Project 21-21762X. FG, KP: This work has received funding from the MUNI Award in Science and Humanities (MUNI/I/1677/2018) of the Grant Agency of Masaryk University. GK: The work on the project leading to this application has received funding from the European Research Council (ERC) under the European Union's Horizon 2020 research and innovation programme (grant agreement No. 741420), from the \'UNKP-20-5 New National Excellence Program of the Ministry of Innovation and Technology from the source of the National Research, Development and Innovation Fund, from Lend\"ulet grant no. 2022-58 and from the J\'anos Bolyai Scholarship of the Hungarian Academy of Sciences.
}
\\
}
\author{Frederik Garbe\thanks{Universit\"at Heidelberg,  Institut f\"ur Informatik, Im Neuenheimer Feld 205, 69120 Heidelberg, Germany. E-mail: {\tt garbe@informatik.uni-heidelberg.de}. Previous affiliation: Faculty of Informatics, Masaryk University, Botanick\'a 68A, 602 00 Brno, Czech Republic.} \and Jan Hladk\'y\thanks{Institute of Computer Science of the Czech Academy of Sciences, Prague, Czech Republic, E-mail: {\tt hladky@cs.cas.cz}} 
\and Gábor Kun\thanks{HUN-REN Alfr\'ed R\'enyi Institute of Mathematics,  Budapest, Hungary and Institute of Mathematics, E\"otv\"os Lor\'and University, Budapest, Hungary, E-mail: {\tt kungabor@renyi.hu}} \and 
Kristýna Pekárková\thanks{Faculty of Informatics, Masaryk University, Brno, Czech Republic, E-mail: {\tt kristyna.pekarkova@mail.muni.cz}}}

\date{}

\maketitle

\begin{abstract}

The theory of limits of permutations leads to limit objects called permutons, which are certain Borel measures on the unit square. 
We prove that permutons avoiding a given permutation of order $k$ have a particularly simple structure. Namely, almost every fiber of the disintegration of the permuton (say, along the x-axis) consists only of atoms, at most $(k-1)$ many, and this bound is sharp. We use this to give a simple proof of the `permutation removal lemma'.
\end{abstract}

\newpage
\section{Introduction}\label{sec:Intro}
The main result of this paper concerns limit properties of pattern-avoiding permutations.
Here, by a \emph{permutation} we mean a bijection $\pi:[n]\to [n]$ for some $n\in\NN$. We write $\mathbb{S}(n)$ for the set of all permutations on $[n]$. The fact that $[n]$ is equipped with the natural order allows us to define the notion of a pattern. Namely, for $k\in [n]$ and for a $k$-element set $K\in\binom{[n]}{k}$ we say that \emph{$K$ induces a pattern $A\in\mathbb{S}(k)$ in $\pi$} if for all $i,j\in [k]$ we have that $A(i)<A(j)$ if and only if the image of the $i$th smallest element of $K$ under $\pi$ is smaller than the image of the $j$th smallest element of $K$ under $\pi$. The \emph{density of $A$ in $\pi$}, denoted by $t(A,\pi)$, is the proportion of $k$-tuples $K$ that induce $A$. We say that $\pi$ is \emph{$A$-avoiding} if $t(A,\pi)=0$. Pattern avoidance is one of the most vivid parts of the combinatorics of permutations, and indeed its treatise spans most of the standard books on permutations~\cite{MR2919720,MR3012380}.
A classical result that can be understood in this way is that of Erd\H os and Szekeres~\cite{esz} on monotone patterns. Most of the questions in the area of pattern avoidance concern a setting in which the pattern is fixed and $n$ is large. A very broad question is how restrictive the property of avoiding a specific pattern is. The most famous result in this direction is due to Marcus and Tardos~\cite{MaTa04}, formerly known as the Stanley--Wilf conjecture, and says, that for any fixed pattern $A$ there exists a constant $c_{A}$ so that the number of $A$-avoiding permutations of any order
$n$ is at most $c_{A}^{n}$. The monographs of B\'ona~\cite{MR2919720} and Kitaev~\cite{MR3012380} contain many further results regarding the number and structure of pattern-avoiding permutations.

Following the success of the theory of limits of dense graphs, Hoppen,
Kohayakawa, Moreira, Ráth, and Sampaio~\cite{HoKo13} introduced a limit theory for finite permutations; the name and the following measure theoretic view on the limit object were introduced by Kr\'al' and Pikhurko~\cite{KrPi13}. Here, we recall only basics needed to state our main result and defer details to Section~\ref{ssec:permutons}. A \emph{permuton} $\Gamma$ is a Borel probability measure on $[0,1]^{2}$ with uniform marginals, that is, $\Gamma(Z\times[0,1])=\Gamma([0,1]\times Z)$ is equal to the Lebesgue measure of $Z$ for any Borel set $Z\subset[0,1]$.
The main feature of the theory of permutons is that they generalise permutations and form a compact metric space with the so-called \emph{rectangular distance}. Furthermore, pattern densities are continuous with respect to this metric. 

We now formally define pattern densities for permutons. Recall that for $A\in \mathbb{S}(k)$ and $\pi\in\mathbb{S}(n)$ with $k\le n$ we have defined the density of $A$ in $\pi$ as the proportion of $k$-tuples of $[n]$ which induce the pattern $A$. To define pattern avoidance for permutons we will work with geometric representations of a pattern $A\in \mathbb{S}(k)$. Let $S_A\subset ([0,1]^2)^k$ be the collection of all sets of $k$ points $(x_1,y_1),(x_2,y_2),\ldots,(x_k,y_k)$ such that for all $1\le i<j\le k$ we have $x_i<x_j$ and further that $y_i<y_j$ if and only if $A(i)<A(j)$. We call each element of $S_A$ a \emph{geometric representation} of $A$. Vice versa each collection $(x_1,y_1),(x_2,y_2),\ldots,(x_k,y_k)$ of $k$ points such that for every $i< j$ we have $x_i< x_j$ and $y_i\neq y_j$ \emph{induces} a unique permutation $A\in \mathbb{S}(k)$ via $A(i)\coloneqq
|\{j\in[k]\mid y_j\leq y_i\}|$. Note that each element of $S_A$ induces $A$. We can now extend the notion of pattern density to permutons. The \emph{density of $A$ in a permuton $\Gamma$} is defined as $t(A,\Gamma)\coloneqq k!\cdot \Gamma^{\otimes k}(S_A)$, where $\Gamma^{\otimes k}$ denotes the $k$-dimensional product measure of $\Gamma$. A probabilistic interpretation of $t(A,\Gamma)$ is as follows. Suppose
that we sample points $\left(x_{1},y_{1}\right),\left(x_{2},y_{2}\right),\ldots,\left(x_{k},y_{k}\right)$
independently at random from the measure $\Gamma$. Then $t(A,\Gamma)$
is the probability that when reading these points from left to right,
their vertical positions are consistent with $A$. In particular, we say that $\Gamma$ is \emph{$A$-avoiding} if $t(A,\Gamma)=0$.%

Let $\mathfrak{M}$ be the space of all finite Borel measures on $[0,1]$.
For a given $\gamma\in\mathfrak{M}$ and $\ell\in\mathbb{N}$ we say
that \emph{$\gamma$} is an \emph{$\le\ell$-molecule} if there exists
a set $M\subset[0,1]$ with $|M|\le\ell$ such that $\gamma\left([0,1]\setminus M\right)=0$. We say that $\gamma$ is an \emph{$\ell$-molecule} if it is an $\le\ell$-molecule but not an $\le(\ell-1)$-molecule.

Our theorem below says that the support of the measure of any pattern-avoiding permuton has a particularly simple one-dimensional structure. This is phrased in terms of the disintegration of a measure, a classical concept which we briefly recall tailored to our setting and refer the reader to Section~10.6 of~\cite{MR2267655} for details. Suppose
that $\left\{ \Lambda_{x}\in\mathfrak{M}\right\} _{x\in[0,1]}$ is a collection of measures.
We now add an integrability condition. A reader not very familiar with measure theory can understand the main point of the paper ignoring this condition just like one can understand quite a bit about real integration without knowing that not every function can be integrated.
Suppose that $\left\{ \Lambda_{x}\in\mathfrak{M}\right\} _{x\in[0,1]}$ satisfies that the map $x\mapsto\Lambda_{x}$ is
measurable with respect to the Borel $\sigma$-algebra induced by the weak convergence topology on $\mathfrak{M}$.
Then we can define a new Borel measure $\Lambda$ on $[0,1]^{2}$
by
\[
\Lambda(B):=\int_{x}\Lambda_{x}\left(\left\{ y\in[0,1]\;|\;(x,y)\in B\right\} \right)\;\D\lambda(x)\quad\text{for each Borel \ensuremath{B\subset[0,1]^{2}}}\;.
\]
Then $\left\{ \Lambda_{x}\right\}_{x\in[0,1]}$ is called a \emph{disintegration
}of $\Lambda$ with respect to the Lebesgue measure $\lambda$. The
Disintegration Theorem tells us that a disintegration always exists
and is unique up to a nullset. We call the measures $\Lambda_{x}$
\emph{fibers}. 

We give two examples which are particularly relevant for us. Firstly, the (unique, up to a nullset) disintegration of the 2-dimensional Lebesgue measure $\lambda^{\otimes 2}$ on $[0,1]^2$ is the collection of 1-dimensional Lebesgue measures, $\left\{ \Lambda_{x}=\lambda\right\}_{x\in[0,1]}$. Secondly, consider the \emph{diagonal permuton} $D$, that is, $D(Z)=\lambda(\{x\in[0,1]\;|\;(x,x)\in Z\})$ for every Borel $Z\subset [0,1]^2$. Then at each $x\in[0,1]$, the fiber of $D$  at $x$ is the Dirac measure $\mathsf{Dirac}_x$. In particular, in the second example, all fibers are $1$-molecules, whereas in the first example they are not $\leq \ell$-molecules for any $\ell$.

We can now state our main result.
\begin{thm}\label{thm:molecule}
Suppose that $A$ is a pattern of order $k$ and $\Gamma$ is an $A$-avoiding permuton. Let us fix a disintegration $\left\{ \Gamma_{x}\right\} _{x\in[0,1]}$ of $\Gamma$. Then almost all fibers of $\Gamma$ are $\le\left(k-1\right)$-molecules.

This bound is optimal: for every pattern $A$ of order $k$ there exists an $A$-free permuton such that almost all of its fibers are $(k-1)$-molecules.
\end{thm}
Theorem~\ref{thm:molecule} is a counterpart to a result of Cooper on finite permutations,~\cite[Theorem~3]{MR2212495}. The language of permutons allows for a simpler formulation and proof. In particular, to prove Theorem~\ref{thm:molecule} we use the Lebesgue density theorem while the proof in~\cite{MR2212495} is based on a regularity lemma for permutations (developed in that paper). This difference is yet another illustration how analytic statements combined with an appropriate theory of limits of discrete structures can substitute combinatorial tools such as regularity lemmas. 
One of the most fascinating demonstrations of this phenomenon is Elek's and Szegedy's approach to hypergraph limits \cite{MR2964622}.

Let us remark that Dole\v{z}al,  M\'{a}th\'{e} and the second author~\cite{DoHlMa17} proved a result in a similar spirit (but using completely different tools) for graphons: every graphon avoiding a specific graph must be countably-partite.

Our construction for the second half of the theorem will be piecewise linear. One may expect that pattern-free permutons are similarly nice. However, this need not be the case as we show in Section~\ref{ssec:nondifferentiable}. 

\begin{figure}[t]
    \centering
    \includegraphics[scale=0.7]{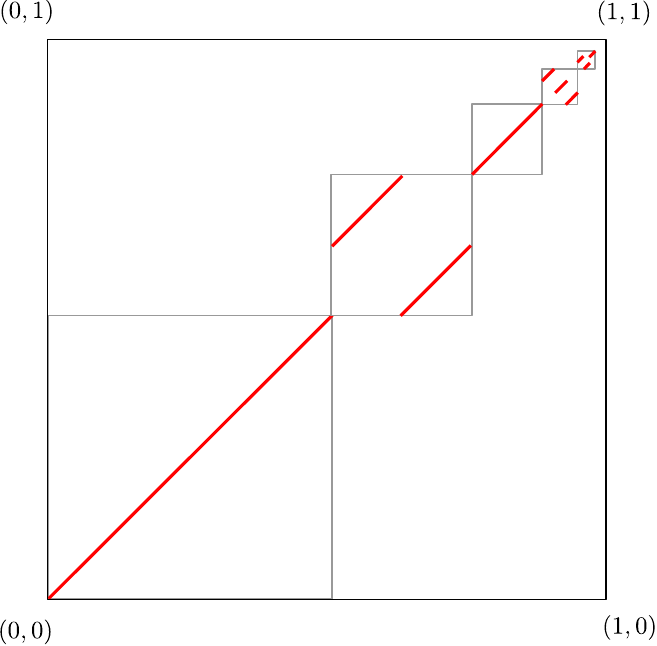}
    \caption{Permuton $\Gamma_{\mathrm all}$ for the sequence $A_1=(12)$, $A_2=(21)$, $A_3=(123)$, $A_4=(321)$, $A_5=(213)$, \ldots.}
	\label{fig:DirectSum}
\end{figure}
A reverse of Theorem~\ref{thm:molecule} does not hold. That is, we can construct a permuton $\Gamma_{\mathrm all}$ such that every disintegration only consists of $1$-molecule fibers, but $\Gamma_{\mathrm all}$ contains every pattern with positive density. For a given pattern $A$ of size $k$ we can construct a blow-up permuton $\Gamma_A$ whose support consists of $k$ increasing segments of equal length ordered according to $A$. Choosing an enumeration $\{A_1,A_2,\cdots\}$ of all countably many patterns we can scale each permuton $\Gamma_{A_k}$ by the factor $\tfrac{1}{2^{k}}$ and form the direct sum $\Gamma_{\mathrm all}=\bigoplus_{k\in\mathbb N} \tfrac{1}{2^{k}}\Gamma_{A_k}$ which just means that we place the scaled $\Gamma_{A_k}$ along the diagonal. See Figure~\ref{fig:DirectSum} for an illustration. Formally, $\Gamma_{\mathrm all}$ is the unique permuton with the support
\[\bigcup_{k\in\mathbb N, i\in[|A_k|]}\left\{\left(\tfrac{i-1}{2^k|A_k|}+x+\sum_{j=1}^{k-1}\tfrac{1}{2^j},\tfrac{A_k(i)-1}{2^k|A_k|}+x+\sum_{j=1}^{k-1}\tfrac{1}{2^j}\right)\mid 0\leq x\leq \tfrac{1}{2^k|A_k|}\right\}\;.\]
Since every vertical or horizontal line intersects exactly one increasing segment in exactly one point every fiber of a disintegration is a $1$-molecule. Furthermore, the occurrence of the scaled copy of $\Gamma_{A_k}$ in $\Gamma$ ensures that the each pattern $A_k$ appears with positive density. 

\subsection{A removal lemma as an application}\label{ssec:remlemma}
As an application, we provide a proof of the following removal lemma. This is a nice example how a result about limit objects can be used to show a result for finite permutations.

\begin{thm}\label{thm:Relocation}
	Let $k\in\NN$ and $A\in\mathbb{S}(k)$. For every $\eps>0$ there exists $\delta>0$ such that the following holds. Suppose that $\pi\in\mathbb{S}(n)$ is a permutation with $t(A,\pi)<\delta$. Then there exists a permutation $\tilde\pi\in\mathbb{S}(n)$ which is $A$-avoiding and for which we have 
\begin{equation}\label{eq:removalchange}
	\sum_{i = 1}^n|\pi(i)-\tilde\pi(i)|<\eps n^2\;.
\end{equation}
\end{thm}
This `permutation removal lemma' was first obtained by Klimo\v sov\'a and Kr\'al' \cite{MR3376446} and reproved by Fox and Wei~\cite{MR3627836,MR3816060}. The former proof gave an Ackermann-type dependency between $\eps$ and $\delta$.\footnote{Note that~\cite{MR3376446} claims a doubly exponential dependency. This claim is wrong as was pointed out in Footnote~2 of~\cite{MR3627836}.} The main upshot of the latter proof was in giving a polynomial dependency between $\eps$ and $\delta$ (with the degree of the polynomial being $2^{2^{O(k)}}$). Our proof is weaker in that it does not give a quantitative dependence between $\eps$ and $\delta$. Its advantage, however, is that it is (with Theorem~\ref{thm:molecule}) almost computation-free, compared to the fairly involved proofs of Klimo\v sov\'a--Kr\'al' and Fox--Wei.
Let us also note that the metric used to quantify~\eqref{eq:removalchange} is called \emph{Spearman’s footrule distance}, and is within  a  factor  of~2 within another, perhaps more common distance, the so-called \emph{Kendall's tau distance}~\cite{DiGr77}.

\section{Notation and preliminaries}\label{sec:preliminaries}

In this section we introduce some basic results from the theory of permutation limits and measure theory. This gives the context of Theorem~\ref{thm:molecule}, although its elementary proof will not make use of these additional results. However, we will use them in the proof of Theorem~\ref{thm:Relocation}. In the following $\lambda$ denotes the Lebesgue measure on $[0,1]$. For an arbitrary measure $\gamma$ on a measure space $X$, we write $\gamma^{\otimes k}$ for its $k$-th power, the product measure on $X^k$.\\

\subsection{Limits of permutations}\label{ssec:permutons}
We recall basic concepts of the theory of permutation limits as introduced in~\cite{HoKo13}. Earlier we already defined a \emph{permuton} as a probability measure on the Borel $\sigma$-algebra on $[0,1]^2$ with the uniform marginal property. The relevant metric for the limit theory of permutations is the \emph{rectangular distance} which is defined for two permutons $\Gamma_1$ and $\Gamma_2$ as
\begin{equation}\label{eq:rectangular}
\dcut(\Gamma_2, \Gamma_2) = \sup_{\substack{S, T \subseteq [0, 1]\\ \textrm{intervals}}} |\Gamma_1(S \times T) - \Gamma_2(S \times T)|\;.
\end{equation}
\begin{remark}\label{rem:nonpermutons}
~\eqref{eq:rectangular} gives a metric for general measures (i.e., not necessarily permutons), as well.
\end{remark}

Note that each finite permutation $\pi\in\mathbb{S}(n)$ has a \emph{permuton representation} $\Psi_\pi$, which is defined as follows. Take the union of rectangles $S\coloneqq\bigcup_{i=1}^n [\frac{i-1}{n},\frac{i}{n})\times [\frac{\pi(i)-1}{n},\frac{\pi(i)}{n})$. Now, for any Borel set $B\subset [0,1]^2$ we define $\Psi_\pi(B)\coloneqq n\cdot \lambda^{\otimes 2} (B\cap S)$.
Using permuton representations we can extend the rectangular distance to all finite permutations. That is, given a permuton $\Gamma$ and finite permutations $\pi_1\in\mathbb{S}(n_1)$ and $\pi_2\in\mathbb{S}(n_2)$ we write $\dcut(\pi_1,\Gamma)\coloneqq\dcut(\Psi_{\pi_1},\Gamma)$ and $\dcut(\pi_1,\pi_2)\coloneqq\dcut(\Psi_{\pi_1},\Psi_{\pi_2})$. The rectangular distance measures differences in global distributions of mass rather than differences in individual  images, which corresponds to Spearman’s footrule distance (see discussion below~\ref{eq:removalchange}). A good illustration of this is the case when $n$ is large and $\pi_1$ and $\pi_2$ are taken uniformly at random in $\mathbb{S}(n)$. Then, with high probability, $\dcut(\pi_1,\pi_2)\approx 0$, even though the images of individual elements of $\pi_1$ and of $\pi_2$ differ substantially. In fact, the 2-dimensional Lebesgue measure is the correct limit object, that is, with high probability, $\dcut(\pi_1,\lambda^{\otimes 2})\approx\dcut(\pi_2,\lambda^{\otimes 2})\approx  0$.

The following compactness result is the key result of the limit theory for permutations. We will use it to pass from finite permutations in Theorem~\ref{thm:Relocation} to a limit permuton to which we can apply Theorem~\ref{thm:molecule}.
\begin{thm}[\cite{HoKo13}]\label{thm:compact}
	The space of permutons is compact with respect to the distance $\dcut$.
\end{thm}

In fact, Theorem~\ref{thm:compact} follows immediately from the fundamental theorem of Prokhorov that the weak topology (recalled in Section~\ref{ssec:LevyProkhorov}) of probability measures on $[0,1]^2$ is compact, since the rectangular distance gives the weak topology, and the closed subset of probability measures with uniform marginals is exactly the set of permutons, hence it is compact.

Furthermore, we will need that the convergence in rectangular distance also induces the convergence in pattern densities.

\begin{thm}[\cite{HoKo13}]\label{thm:densitiescontinuous}
 Suppose that $\pi_1,\pi_2,\ldots$ is a sequence of permutations of growing order that converges to a permuton $\Gamma$ in the rectangular distance. Then for any $k\in\NN$ and any $A\in\mathbb{S}(k)$ we have $\lim_{n\rightarrow \infty} t(A,\pi_n)=t(A,\Gamma)$.
\end{thm}

\subsection{The Lévy-Prokhorov metric}\label{ssec:LevyProkhorov}
We will make use of the Lévy-Prokhorov metric to compare probability measures on $[0,1]$.

\begin{defi}[Lévy-Prokhorov metric]
For $\eps>0$ and a set $A\subseteq[0,1]$ we define the \emph{$\eps$-neighbourhood of $A$} by
$$A^{\uparrow\eps}=[0,1]\cap\bigcup_{q\in A}(q-\eps,q+\eps)\;.$$
Let $\alpha$ and $\beta$ be two Borel probability measures on $[0,1]$. The \emph{Lévy-Prokhorov distance} between $\alpha$ and $\beta$ is defined by
$$\dlp(\alpha, \beta) = \inf\{ \eps > 0 \mid \forall A \in \mathcal{B}([0,1])\;:\; \alpha(A) \leq \beta(A^{\uparrow\eps}) + \eps,\; \beta(A) \leq \alpha(A^{\uparrow\eps}) + \eps\}\;.$$
\end{defi}

Recall that a sequence of Borel probability measures $\alpha_1,\alpha_2,\ldots$ on $[0,1]$ \emph{converges weakly} to a measure $\alpha$ if for every Borel set $A\subset[0,1]$ we have $\lim_n \alpha_n (A)=\alpha(A)$. It is well-known that the Lévy-Prokhorov metric is a metrization of the topology from weak convergence. It is also well-known that the Lévy-Prokhorov metric is separable, that is, there exists a countable set $\{m_1,m_2,\ldots\}$ of Borel probability measures on $[0,1]$ so that for every Borel probability measure $\gamma$ on $[0,1]$ and every $\eps>0$ there exists an index $i\in\NN$ so that $\dlp(m_i,\gamma)<\eps$. Last, recall that for any metric space, the properties of being separable, having the property of Lindel\"of and being second-countable are all equivalent.\footnote{We do not give definitions since the only way we use them is that we feed them into another theorem.} In particular, we will use that the Lévy-Prokhorov metric gives a second-countable space.

\begin{lem}\label{lem:measure}
    Let $\Gamma$ be a probability measure on $[0,1]^2$.
    For all $\delta > 0$, there exists $h \in \NN$ and a partition 
    $[0,1] = J_1 \cup J_2 \cup \ldots \cup J_h$ of intervals
    such that for a disintegration $\{\Gamma_x\}_{x \in[0,1]}$ of $\Gamma$
    there exists $X \subset [0,1]$ with $\lambda(X) < \delta$
    such that $\forall i \in [h], x,y \in J_i \setminus X$: $\dlp(\Gamma_x, \Gamma_y) < \delta$.
\end{lem}
\begin{proof}
	We shall use Lusin's Theorem in the following form, see~\cite[Theorem 17.12]{MR1321597}. Suppose that $M$ is a second-countable topological space.
	Suppose that $f:[0,1]\rightarrow M$ is measurable. Then for every $\delta>0$ there exists an open set $X\subset [0,1]$ with $\lambda(X)<\delta$ so that $f_{\restriction [0,1]\setminus X}$ is continuous.
	
	We are in this setting of Lusin's Theorem as the space of all Borel probability measures on $[0,1]$ equipped with the Lévy-Prokhorov metric is second-countable, and the map $f:x\mapsto \Gamma_x$ is Borel. We use it with the same parameter as in the above statement. Since $[0,1]\setminus X$ is compact and the function $f$ is continuous on it, it is also uniformly continuous. That is, for the given $\delta$ there exists $\alpha>0$ so that for any $x,y\in [0,1]\setminus X$ with $|x-y|\le \alpha$ we have $\dlp(\Gamma_x,\Gamma_y)<\delta$.
	The lemma follows by choosing $(J_i)_{i\in [h]}$ to be consecutive intervals of length $\alpha$.
\end{proof}

\section{Proof of Theorem~\ref{thm:molecule}}\label{sec:pfmolecule}
\subsection{Proof of the upper bound}
Suppose for the sake of contradiction  that the assertion does not hold. Then there exists a pattern $A$ of order $k$, an $A$-avoiding permuton $\Gamma$ together with a disintegration $\{\Gamma_x\}_{x\in[0,1]}$ and a set $X\subset[0,1]$ with $\lambda(X)>0$ such that for every $x\in X$ the fiber $\Gamma_x$ is not a $\le\left(k-1\right)$-molecule. 

\begin{figure}
    \centering
    \includegraphics[scale=1]{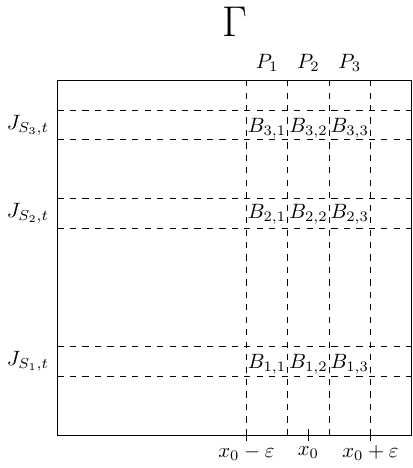}
	\caption{An example for the notation in Section~\ref{sec:pfmolecule} with $k=3$. Each of the boxes $B_{i,j}$, $i,j\in[3]$, contains some positive mass of $\Gamma$, hence any $3\times 3$-pattern occurs with positive density.}
	\label{fig:pfmolecule}
\end{figure}
Figure~\ref{fig:pfmolecule} depicts the setting we introduce below.
For $t\in\mathbb N$ and $m\in[t-1]$ we define $J_{m,t}\coloneqq[(m-1)/t,m/t)$ and $J_{t,t}\coloneqq[(t-1)/t,1]$. For $S\in\binom{[t]}{k}$ and $i\in[k]$, write $S_i$ for the $i$th smallest element of $S$. For each $S\in\binom{[t]}{k}$, set $X_{S,t}\coloneqq\{x\in X\mid\Gamma_{x}(J_{S_i,t})>0\;\forall i\in[k]\}$. Note that for $x\in X$, since $\Gamma_x$ is not $\le\left(k-1\right)$-molecule, we have that there exist $t\in\mathbb N$ and $S\in\binom{[t]}{k}$ such that $x\in X_{S,t}$. Hence $X=\bigcup_{t\in\mathbb N, S\in\binom{[t]}{k}}X_{S,t}$. Recalling that $\lambda(X)>0$, it follows from the countable subadditivity of measures that there exist $t\in\mathbb N$ and $S\in\binom{[t]}{k}$ such that $\lambda(X_{S,t})>0$. By Lebesgue's density theorem, we can find a density point $x_0$ of $X_{S,t}$  and choose $\eps>0$ small enough such that $\lambda([x_0-\eps,x_0+\eps]\cap X_{S,t})>2\varepsilon (k-1)/k$. Then for $i\in[k]$ define $P_i\coloneqq [x_0-\varepsilon+2\varepsilon\cdot (i-1)/k,x_0-\varepsilon+2\varepsilon\cdot i/k)$ and note that $\lambda(X_{S,t}\cap P_i)>0$. Finally, define $B_{i,j}\coloneqq P_i\times J_{S_j,t}$, for $i,j\in[k]$, and observe that every $p\in\prod_{i\in[k]} B_{i,A(i)}$ induces $A$ and therefore
\[t(A,\Gamma)\geq\Gamma^{\otimes k}\left(\prod_{i\in[k]} B_{i,A(i)}\right)\geq\prod_{i\in[k]}\left(\int_{x\in P_i\cap X_{S,t}}\Gamma_x(J_{S_{A(i)},t})\right)>0\;,\]
which is a contradiction.

\subsection{Optimality}\label{ssec:optimalitymolecules}
Given a pattern $A\in\mathbb{S}(k)$, we will construct an $A$-free permuton whose fibers with respect to a disintegration along the x-axis are almost all $(k-1)$-molecules. The role of the two coordinates is exchangeable, and it turns out that it is notationally simpler to construct an $A^{-1}$-free
permuton whose horizontal fibers are $(k-1)$-molecules.

Set $\pi=A^{-1}$. We consider the piecewise linear function $L:[0,1]\to[0,1]$ consisting of $k-1$ pieces of slope $k-1$ or $1-k$ satisfying
for $i=1, \dots ,k-1$ that $L(\frac{i-1}{k-1})=0$ and $L(\frac{i}{k-1})=1$ if $\pi(i)>\pi(i+1)$, while $L(\frac{i-1}{k-1})=1$ and $L(\frac{i}{k-1})=0$ if $\pi(i)<\pi(i+1)$.
(At the endpoints this might not be well-defined: in this case we choose the value defined on the interval on the left.)
An example is given in Figure~\ref{fig:Piecewise}.
\begin{figure}[t]
    \centering
	\includegraphics[scale=0.7]{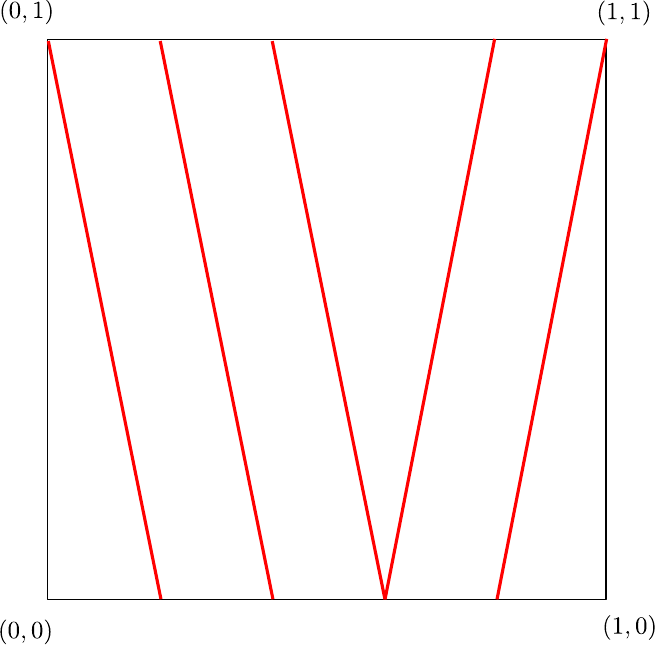}
	\caption{Construction of a piecewise linear function $L$ in Section~\ref{ssec:optimalitymolecules} for $\pi=(145632)$.}
	\label{fig:Piecewise}
\end{figure}

Putting an appropriate multiple of the 1-dimensional Lebesgue measure on the graph of the function $L$ we obtain a permuton (that is, both uniform marginal properties are satisfied), which we call $\Gamma$. Firstly, almost all horizontal fibers of $\Gamma$ are $(k-1)$-molecules. So, it remains to argue that $t(\pi,\Gamma)=0$. To this end we prove the following.
\begin{unnumberedclm}
Consider $\ell\in[k]$ and $0\leq x_1 < \dots < x_{\ell} \leq 1$ such that $x_i\not\in\{\frac{j}{k-1}:j\in\ZZ\}$ for all $i\in [\ell]$. If the set of points $\{ (x_1,L(x_1))\ldots,(x_{\ell},L(x_{\ell})) \}$ induces the pattern given by the first $l$ entries of $\pi$ then $x_{\ell} > \frac{\ell-1}{k-1}$.
\end{unnumberedclm}
Observe that the above claim indeed proves that $t(\pi,\Gamma)=0$, since it asserts that we cannot find geometric representations of $\pi$ in the support of $\Gamma$, except for the nullset of $k$-tuples where at least one of the coordinates $x_i$ lies in $\{\frac{j}{k-1}:j\in\ZZ\}$.

Let us now prove the claim by induction on $\ell$. The base case $\ell=1$ is obvious. Assume that the claim holds for $\ell-1$ and, in particular, for the $(\ell-1)$-tuple $(x_1,\ldots,x_{\ell-1})$. If $x_{\ell-1} > \frac{\ell-1}{k-1}$ then we are done, since $x_{\ell}>x_{\ell-1}$. It remains to handle the case $x_{\ell-1} \in (\frac{\ell-2}{k-1},\frac{\ell-1}{k-1})$. Then $x_{\ell} \notin (\frac{\ell-2}{k-1},\frac{\ell-1}{k-1})$, since $x_{\ell} >x_{\ell-1}$, and if 
$\pi(\ell-1)>\pi(\ell)$ then $L$ is monotone increasing on this interval, while monotone decreasing otherwise. This completes the proof the theorem.

Note that while a.e. vertical fiber of the construction is a $(k-1)$-molecule, the horizontal fibers are $1$-molecules, i.e., their support is a single atom of measure~1. It would be interesting to see, which lower bounds on the size of the support of a.e. horizontal and vertical fiber would guarantee that a given pattern has positive density.

\section{Proof of Theorem~\ref{thm:Relocation}}
Let us first give an overview of the proof. By using the compactness of the space of permutons it suffices to show for a sequence $\pi_1,\pi_2,\dots$ of not necessarily $A$-avoiding permutations converging to an $A$-free permuton $\Gamma$ that for $n$ large enough we can find similar but $A$-avoiding permutations $\tilde\pi_n$. Suppose that $\pi_n$ is close to $\Gamma$ in the rectangular distance. Fix a disintegration $\{\Gamma_x\}_{x \in[0,1]}$ of $\Gamma$. We generate an $n$-tuple of numbers uniformly at random in $[0,1]$, and read them in increasing order as $x_1<x_2<\cdots<x_n$. We set $y_i \coloneqq (\pi_n(i)-0.5)/n$. The set of points $\{(x_1,y_1)$, \ldots, $(x_n,y_n)\}$ is a geometric representation of $\pi_n$, and the uniform probability measure supported on this set is close to $\Gamma$ in the rectangular distance. 
The key idea is that almost any $n$-tuple of points $(\tilde{x}_1,\tilde{y}_1)$, \ldots, $(\tilde{x}_n,\tilde{y}_n)$, whose x- and y-coordinates are pairwise distinct and where all the points lie in the support of the measure $\Gamma$, is a geometric representation of an $A$-free permutation. We shall use this with the choice $\tilde{x}_1=x_1$, \ldots, $\tilde{x}_n=x_n$. We want to alter the y-coordinates of the points $(x_i,y_i)$ such that the new $\tilde{y}_i$ lies in the support of the measure $\Gamma_{x_i}$, which is essentially equivalent to the point $(x_i,\tilde{y}_i)$ lying in the support of $\Gamma$.

Here is where Theorem~\ref{thm:molecule} comes into play. We know that almost every $\Gamma_{x_i}$ is a collection of at most $k-1$ atoms. Since $\pi_n$ is close to $\Gamma$ in the rectangular distance, we can also deduce that typically $\Gamma_{x_i}$ must have positive mass around $y_i$. Therefore, there exists an atom $\tilde{y}_i$ of $\Gamma_{x_i}$ which is close to $y_i$. The points $(x_i,\tilde{y}_i)$ then induce our new permutation $\tilde{\pi}_n$. See Figure~\ref{fig:Example}.

\begin{figure}[t]
    \centering
    \includegraphics[width=\linewidth]{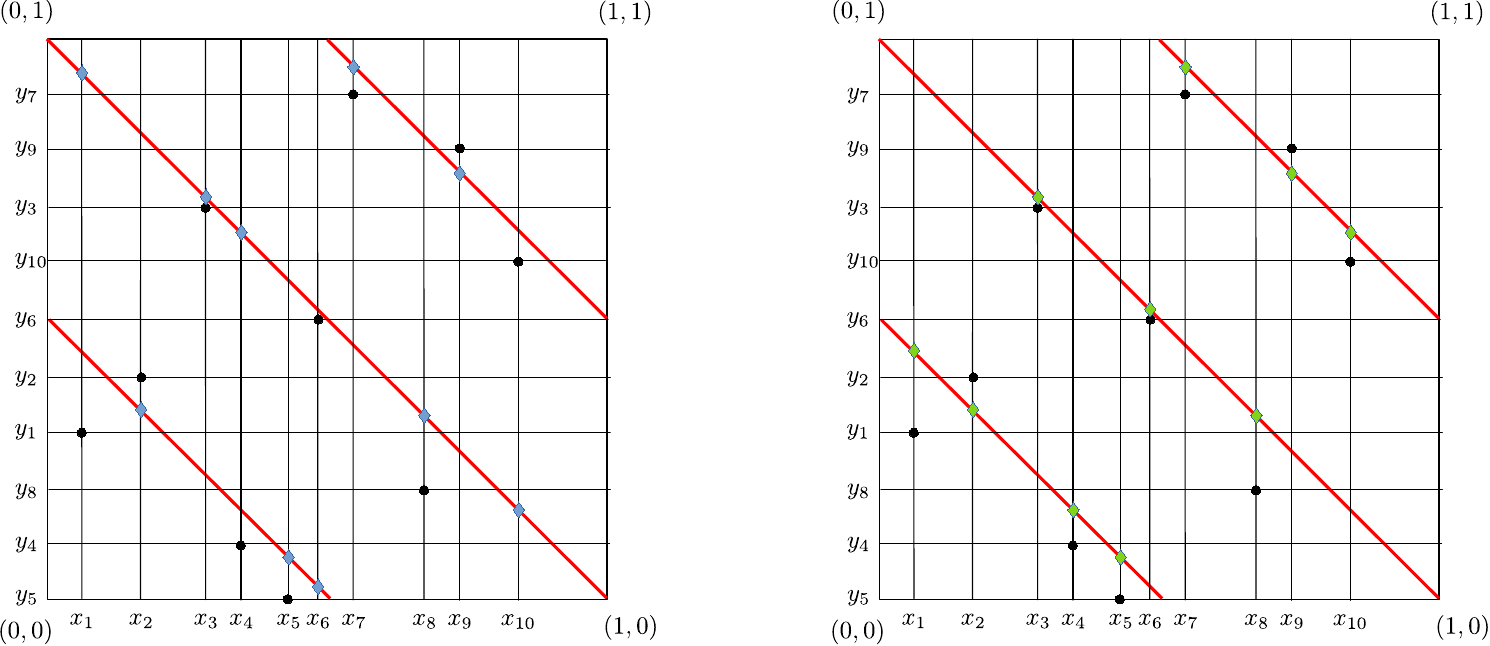}
    \caption{An example described in Section~\ref{ssec:remlemma}. In red, the support of a permuton $\Gamma$ which is free of the pattern $A=(1234)$. Each fiber of the disintegration of $\Gamma$ consists of two atoms, each of weight $0.5$. Black dots show points $(x_i,y_i)$ representing a permutation $\pi_{10}$ which is close to $\Gamma$. \emph{Left:} Grey diamonds show positions of points $(x_i,\tilde{y}_i)$ which represent an $A$-free permutation which might be an outcome of a random resnapping on the fibers $\Gamma_{x_i}$. The issue is that $|y_i-\tilde{y}_i|$ is big for $i=1,4,6,10$. \emph{Right:} Green diamonds correspond to a resnapping to the closest atom on each fiber $\Gamma_{x_i}$.}
	\label{fig:Example}
\end{figure}

\bigskip
We can now proceed with details. We shall prove Theorem~\ref{thm:Relocation} in the following form.

\begin{thm}\label{thm:weprove}
		Let $k\in\NN$ and $A\in\mathbb{S}(k)$. Suppose that we have a sequence of permutations $(\pi_n\in\mathbb{S}(n))_n$ converging to an $A$-free permuton $\Gamma$ in the rectangular distance. Then for every $\eps>0$ there exists $n_0 \in \NN$ such that for every $n > n_0$ there exists an $A$-free permutation $\tilde{\pi}_n\in \mathbb{S}(n)$ such that 
	\begin{equation}\label{milujuvzorecky}
	\sum_{i = 1}^{n}|\pi_n(i)-\tilde{\pi}_n(i)| < \eps n^2\;.
	\end{equation}
\end{thm}
\begin{proof}[Deducing Theorem~\ref{thm:Relocation} from Theorem~\ref{thm:weprove}]
	Suppose that Theorem~\ref{thm:Relocation} does not hold. That is, for a fixed pattern $A$, there exists a number $\eps>0$ and a sequence of permutations $\pi_n$ with $t(A,\pi_n)\to 0$ for which we cannot find similar but $A$-free permutations. By compactness of the space of permutons (Theorem~\ref{thm:compact}), there exists a subsequence $(\pi_{n_\ell})$ which converges to a certain permuton $\Gamma$ in the rectangular distance. By continuity of densities (Theorem~\ref{thm:densitiescontinuous}), $\Gamma$ is $A$-free. This setting contradicts Theorem~\ref{thm:weprove}, as was needed.
\end{proof}

\begin{proof}[Proof of Theorem~\ref{thm:weprove}]
Given $\pi_n\overset{\dcut}{\rightarrow}\Gamma$, where $\Gamma$ is $A$-free, and $\eps>0$ we fix constants $\zeta,\delta>0$ and $h,n_0\in\mathbb N$ according to the following hierarchy
\[\tfrac{1}{n_0} \ll \zeta \ll \tfrac{1}{h} \ll \delta \ll \eps\;.\]

First, we apply Lemma~\ref{lem:measure} to $\Gamma$ and $\delta$ resulting in $h\in\mathbb N$, a partition $[0,1]=\bigcup_{i\in[h]}J_i$ into intervals, a set $X_1\subseteq[0,1]$ with $\lambda(X_1)<\delta$, and a disintegration $\{\Gamma_x\}_{x \in[0,1]}$ of $\Gamma$ such that for every $i\in[h]$ and $x,y\in J_i\setminus X_1$ we have
\begin{equation}\label{eq:distint}
\dlp(\Gamma_x, \Gamma_y) < \delta\;.
\end{equation}

Furthermore, since $\Gamma$ is $A$-free and by Theorem~\ref{thm:molecule} we can fix a null set $X_2\subseteq [0,1]$ such that $\Gamma_x$ is $\le\left(k-1\right)$-molecule for every $x\in [0,1]\setminus X_2$. We then set $X\coloneqq X_1\cup X_2$.

Now, fix a permutation $\pi_n$ with $n\geq n_0$. We want to show that there exists an $A$-free permutation $\tilde{\pi}_n\in \mathbb{S}(n)$ fulfilling~\eqref{milujuvzorecky}. Note that by the constant hierarchy we can choose $n_0$ large enough such that we may assume that
\begin{equation}\label{eq:cutdist}
\dcut(\pi_n, \Gamma)  < \zeta/2\;.
\end{equation}

Set $T\coloneqq\{i\in[n]\mid \lambda(((i-1)/n,i/n)\setminus X)>0\}$. Note that $|T|>(1-\delta) n$, as $\lambda(X)<\delta$. We now generate a random $n$-tuple of points in the unit interval by choosing $x_i$ uniformly at random from $((i-1)/n,i/n)\setminus X$ for every $i\in T$. For $i\in [n]\setminus T$ we choose $x_i$ uniformly at random from $((i-1)/n,i/n)\setminus X_2$. Furthermore, we set 
\begin{equation}\label{eq:whatisy}
y_i\coloneqq (\pi_n(i)-0.5)/n
\end{equation}
and denote the measure which has mass $1/n$ on each point $(x_i,y_i)$ by $\Delta$, $\Delta=\frac1n\sum_{i\in[n]}\mathsf{Dirac}_{(x_i,y_i)}$. While $\Delta$ is not a permuton, it can be thought of as a perturbation of the permuton representation $\Psi_{\pi_n}$ in which each rectangle $[\frac{i-1}{n},\frac{i}{n})\times [\frac{\pi_n(i)-1}{n},\frac{\pi_n(i)}{n})$ is transformed to a single point of the same measure and within that rectangle. In particular $\dcut(\Delta,\pi_n)\le \tfrac{2}{n}$. From the triangle inequality (c.f. Remark~\ref{rem:nonpermutons}) and~\eqref{eq:cutdist} we get
\begin{equation}\label{eq:cutdistlepsi}
\dcut(\Delta,\Gamma)\leq\dcut(\Delta,\pi_n)+\dcut(\pi_n,\Gamma)\leq\tfrac{2}{n}+\zeta/2<\zeta\;.
\end{equation}

\begin{clm}\label{clm:sureprop} With probability $1$ with respect to the choice of the random points $(x_i)_{i\in [n]}$ the following holds.
\begin{enumerate}[label=(\roman*)]
    \item\label{clm:Afree} If $(\tilde{y}_i\in[0,1])_n$ is collection of pairwise distinct points such that for each $i\in[n]$, $\tilde{y}_i$ is an atom on $\Gamma_{x_i}$ then the permutation induced by $((x_i,\tilde{y}_i))_{i\in[n]}$ is $A$-free.
    \item\label{clm:atomsdis} If $\Gamma_{x_i}(y)>0$ for some $i\in[n]$ and $y\in[0,1]$, then $\Gamma_{x_j}(y)=0$ for every $i\neq j$.
\end{enumerate}
\end{clm}

\begin{proof}[Proof of Claim]
For proving \ref{clm:Afree} suppose for the sake of contradiction that there exists $Q\subseteq[0,1]^n$ with $\lambda^{\otimes n}(Q)>0$ such that for every $x\in Q$ there exists $\tilde{y}(x)\in[0,1]^n$ such that $\prod_{i=1}^n\Gamma_{x_i}(\tilde{y}(x)_i)>0$ and such that $(x_i,\tilde{y}(x)_i)_{i\in[n]}$ contains the pattern $A$. Furthermore, we can assume that the map $x\mapsto \tilde{y}(x)$ is measurable. Then the probability to sample the pattern $A$ from $\Gamma$ is at least
\[\int_{Q}\prod_{i=1}^n\Gamma_{x_i}(\tilde{y}(x)_i)\D\lambda^{\otimes n}(x)>0\;,\]
contradicting that $\Gamma$ is $A$-free.

For proving \ref{clm:atomsdis} suppose for the sake of contradiction that there exists $y\in[0,1]$ such that for $Q_y\coloneqq\{x\in[0,1]\mid \Gamma_x(y)>0\}$ holds that $\lambda(Q_y)>0$. By continuity of the Lebesgue measure there exists $\beta>0$ and $Q_y'\subseteq Q_y$ with $\lambda(Q_y')>0$ such that for every $x\in Q_y'$ we have $\Gamma_{x}(y)>\beta$. Then for an arbitrary $\alpha>0$ we have
\[\Gamma([0,1]\times[y-\alpha,y+\alpha])\geq \beta\lambda(Q_y')\;,\]
which for $\alpha$ small enough contradicts the uniform marginal property of $\Gamma$.
\end{proof}

We fix an outcome of our random selection of $x_1<\cdots<x_n$ such that the above properties hold. Below, we abbreviate measures of intervals $[a,b]$ on fibers $\Gamma_{x_i}$ as $\Gamma_{x_i}(a,b):=\Gamma_{x_i}([a,b])$.

\begin{clm}\label{clm:posmass} For all but at most $3\sqrt{\delta} n$ many $i\in [n]$ it holds that $\Gamma_{x_i}(y_i - 2\sqrt{\delta}, y_i + 2\sqrt{\delta}) > \delta$.
\end{clm}
   
\begin{proof}[Proof of Claim] 
    We call a point $(x_i,y_i)$ {\it bad} if $\Gamma_{x_i}(y_i - 2\sqrt{\delta} , y_i + 2\sqrt{\delta}) \leq \delta$ and let $Y$ be the set of such bad points. We denote by $J(x)$ the interval $J_j$ (where $j\in[h]$) such that $x\in J_j$. Consider a system of rectangles $\mathcal{C}=\{J(x_i)\times [y_i - \sqrt{\delta} , y_i + \sqrt{\delta}] \mid (x_i,y_i) \text{ bad} \}$ covering $Y$ in $[0,1]^2$. 
    Choose a minimal system $\mathcal{M}\subseteq\mathcal{C}$ which covers $\bigcup\mathcal{C}$. For $j\in[h]$, let $\mathcal{M}_j\subset \mathcal{M}$ be the collection of sets of the form $J_j\times U$. By the minimality, we have 
    \begin{equation}\label{eq:Msmall}
    |\mathcal{M}_j|\leq 2/(2\sqrt{\delta})=1/\sqrt{\delta}\;.
   \end{equation}
 Therefore, $\bigcup\mathcal{C}$ can be written as the disjoint union of at most $h/\sqrt{\delta}$ many rectangles. By using~\eqref{eq:cutdistlepsi} we get
    \begin{equation}\label{eq:milujubuchtu}
        \left|\Gamma(\cup\mathcal{C})-\Delta(\cup\mathcal{C})\right|<\zeta h/\sqrt{\delta}\;.
    \end{equation}
    
For every rectangle $J(x_i)\times [y_i - \sqrt{\delta} , y_i + \sqrt{\delta}]=J_j\times B=M\in\mathcal{M}_j$ we have that
    \begin{align*}
        \Gamma(M \setminus (X \times [0,1])) 
        &= \int_{J_j\setminus X}\Gamma_x(B)\D\lambda(x)
        \overset{\eqref{eq:distint}}{\leq} \lambda(J_j)\cdot(\Gamma_{x_i}(B^{\uparrow \delta})+\delta)\;.
    \end{align*}
In particular, since $(x_i,y_i)$ is bad we get
\begin{equation}
 \label{eq:milujubuchty}
\Gamma(M \setminus (X \times [0,1])) \leq 2\delta\lambda(J_j)\;.
\end{equation}
    Thus, the triangle inequality implies that
    \begin{align*}
        \Delta(\cup\mathcal{C}) &\leq \Gamma(\cup\mathcal{C}\setminus(X\times[0,1]))+ \Gamma(X\times[0,1]) + |\Gamma(\cup\mathcal{C}) - \Delta(\cup\mathcal{C})|\\
        &\overset{\eqref{eq:milujubuchtu}}{\leq} \sum_{j\in[h]}\sum_{M\in\mathcal{M}_j}\Gamma(M\setminus(X\times[0,1]))+\delta+\zeta h/\sqrt{\delta}\\
        &\overset{\eqref{eq:Msmall},\eqref{eq:milujubuchty}}{\leq} \sum_{j\in[h]} \tfrac{1}{\sqrt{\delta}}\cdot 2\delta\lambda(J_j)+\sqrt{\delta}\\
        &\leq 3\sqrt{\delta}\;,
    \end{align*}
    and therefore $|Y|\le\Delta(\cup\mathcal{C})n\leq 3\sqrt{\delta}n$.
\end{proof}

Now set $S\coloneqq\{i\in T \mid \Gamma_{x_i}(y_i - 2\sqrt{\delta}, y_i + 2\sqrt{\delta})>\delta\}$. Observe that by Claim~\ref{clm:posmass} and the fact that $|T|\geq(1-\delta)n$ we have $|S|\geq n-3\sqrt{\delta}n-\delta n \geq (1-4\sqrt{\delta})n$. For every $i\in S$, since $\Gamma_{x_i}$ is $\le\left(k-1\right)$-molecule there exists $\tilde{y}_i\in [y_i - 2\sqrt{\delta}, y_i + 2\sqrt{\delta}]$ such that $\Gamma_{x_i}(\tilde{y}_i)>0$. For every $i\in[n]\setminus S$ choose an arbitrary atom $\tilde{y}_i$ of $\Gamma_{x_i}$ which is possible, as $x_i\notin X_2$ for every $i\in[n]$. By Claim~\ref{clm:sureprop}\ref{clm:atomsdis} we have that the points $(x_i,\tilde{y}_i)_{i\in[n]}$ induce a permutation which we denote by $\tilde{\pi}_n$. Furthermore, by Claim~\ref{clm:sureprop}\ref{clm:Afree} $\tilde{\pi}_n$ is $A$-free. It remains to establish the key property~\eqref{milujuvzorecky}. We split the summands $\sum_{i = 1}^{n}|\pi_n(i)-\tilde{\pi}_n(i)|$ according to whether $i\in S$ or $i\notin S$. In the latter case, it is enough to use the trivial bound $|\pi_n(i)-\tilde{\pi}_n(i)|\le n$. So, we need to have a good bound for $|\pi_n(i)-\tilde{\pi}_n(i)|$ when $i\in S$. Recall~\eqref{eq:whatisy} which tells us that the spacing in the y-coordinates of the geometric representation of $\pi_n$ was exactly $\frac1n$. Thus, if we change each point $(y_j)_{j\in S}$ by at most $2\sqrt{\delta}$, the point $y_i$ changes its relative position\footnote{that is `bigger than/smaller than'} with at most $4\sqrt{\delta}n$ other points $(y_j)_{j\in S}$. Taking into account also elements which are not in $S$, we get $|\pi_n(i)-\tilde{\pi}_n(i)|\le 4\sqrt{\delta}n+|[n]\setminus S|$. Putting all this together,
\[
\sum_{i = 1}^{n}|\pi_n(i)-\tilde{\pi}_n(i)|\leq \sum_{i\in S}(4\sqrt{\delta}n+|[n]\setminus S|)+\sum_{i\in[n]\setminus S} n\leq (4\sqrt{\delta}+4\sqrt{\delta})n^2+4\sqrt{\delta}n^2<\eps n^2\;.
\]

\end{proof}

\section{Discussion and further questions}\label{sec:discussion}

\subsection{Nondifferentiable pattern-avoiding permuton}\label{ssec:nondifferentiable}
Theorem~\ref{thm:molecule} together with Lusin's theorem tells us that we can think of the support of a pattern-avoiding permuton as a union of graphs of partial functions that are continuous on a set whose complement is of arbitrary small positive measure. The following example shows that continuity cannot be strengthened to differentiability. More precisely, there exists a function $f: [0,1] \rightarrow [0,1]$ whose restriction to any subset of $[0,1]$ of positive measure is not differentiable and the permuton supported on the graph of this function is avoiding the pattern $(3142)$.

Consider the quaternary expansion (with digits $0,1,2$ and $3$) of the numbers in $[0,1]$. This is well-defined for all, but countably many numbers, the rational numbers whose denominator is a power of two. We will ignore this countable set.
Let $f$ be the function that swaps $1$ and $2$ for every coordinate. This function is injective (up to a nullset) and measure preserving, hence its graph is the support of a permuton.

First we argue that the graph of $f$ is free of $(3142)$.
Suppose for a contradiction that there are four numbers $x_1<x_2<x_3<x_4$ such that $f(x_2) < f(x_4) < f(x_1) < f(x_3)$. Consider the first digit where $x_1$ and $x_2$ differ. Clearly this digit should be $1$ for $x_1$ and $2$ for $x_2$. The number $x_4$ cannot differ in an earlier digit, since $f(x_2) < f(x_4) < f(x_1)$. Hence $x_3$ being less than $x_4$ can also not differ in an earlier digit. We will get a contradiction when checking the critical digit: $f(x_3)>f(x_1), x_3>x_2$ implies that this digit of $x_3$ should be $3$, but then $x_4$ also has digit $3$ here contradicting 
$f(x_1)>f(x_4)$.

It remains to show that the restriction of $f$ to any set of positive measure cannot be differentiable: if there was such a set then on a subset of positive measure for every pair the inequality $|\frac{f(x)-f(y)}{x-y}-f'(x)|<\frac{1}{2}$ would hold.
However, given a natural number $n$, $i \in \{-3, -2, -1, 1,2,3\}$ and real numbers $x, x+i4^{-n}$ agreeing until the first $(n-1)$ digits, note that $\frac{f(x)-f(x+i4^{-n})}{i4^{-n}} \in \{ -2, -1, -\frac{1}{2}, \frac{1}{2}, 1, 2 \}$, and this value depends on $i$ and the $n$th digit of $x$ (but not on $n$), and for every $x$ and $n$ we get three different values (depending on $i$). By the Lebesgue density theorem any set of positive measure contains more than three quarter of a small interval, so we can find four numbers in the set agreeing in all, but one digit. Hence the restriction of $f$ to any set of positive measure cannot be differentiable.

\subsection{Lower bounds on the Stanley--Wilf constant}
For a permutation $A$, let $\mathbb{S}(A;n)\subset \mathbb{S}(n)$ be the set of $A$-avoiding permutations of order $n$. The Stanley--Wilf constant is defined as $c_A:=\lim_n \sqrt[n]{|\mathbb{S}(A;n)|}$. The exact value of the Stanley--Wilf constant is only known for some permutations and permutation classes (see~\cite[Chapter~4]{MR2919720}): trivially for $A\in\mathbb{S}(2)$ we have $c_A=1$, for all $A\in\mathbb{S}(3)$ we have $c_A=4$, for each $k\in\mathbb{N}$ we have $c_{\mathsf{identity}_k}=(k-1)^2$, and then the Stanley--Wilf constant is known for several sporadic examples, including $c_{(1342)}=8$, $c_{(12453)}=9+4\sqrt{2}$.

Consider an $A$-free permuton $\Gamma$. We can use $\Gamma$ to generate many different $A$-free permutations of order $n$. Let us illustrate this with the $(1234)$-avoiding permuton $\Gamma$ from Figure~\ref{fig:Example}. Fix the numbers $x_1<\ldots<x_n$ to be equidistant in $[0,1]$. Each fiber $\Gamma_{x_i}$ has exactly two atoms. So, we choose $y_i$ to be one of these two atoms. We have $2^n$ options in total of generating a geometric configuration. It can be shown that most of these choices represent different permutations, hence we obtain $c_{(1234)}\ge 2$.

\subsection{Permuton limits of typical pattern-avoiding permutations}
Consider the sequence of independent random permutations $\pi_1,\pi_2,\ldots$, where $\pi_i$ is taken from $\mathbb{S}(i)$ uniformly at random. It is a well-known fact that such a sequence converges in the rectangular distance almost surely to the Lebesgue measure on $[0,1]^2$. What if we condition $\pi_i$ to be free of a given pattern $A$?
\begin{qu}
\label{pr:typicalfreepermuton}
Consider a pattern $A\in\mathbb{S}(k)$. Suppose that $\pi_1,\pi_2,\ldots$ is a sequence of independent random permutations, where $\pi_i$ is taken from $\mathbb{S}(A;i)$ uniformly at random. Does this sequence converge in the rectangular distance almost surely? What is the limit permuton $\Delta_A$?
\end{qu}
Theorem~\ref{thm:molecule} asserts that the hypothetical permuton $\Delta_A$ is supported on the union of the graphs of finitely many functions.

Joshua Cooper~\cite{Cooper:question} asked for two patterns $A$ and $B$ the expected density of $B$ in a permutation chosen uniformly at random from $\mathbb{S}(A;n)$. If the answer to Question~\ref{pr:typicalfreepermuton} is positive then it is easy to show that this expected density converges to $t(B,\Delta_A)$.

In a recent work Borga, Das, Mukherjee and Winkler~\cite{BDMW22} have introduced a Gibbs permutation model that might help to approach our question. This model does not include the uniform distribution on permutations avoiding a fixed pattern $A$, but it allows to give an exponentially small probability w.r.t. the density of $A$ to a permutation.

Question~\ref{pr:typicalfreepermuton} is trivial for $A\in\mathbb{S}(2)$. For $A=(132)$, Section~5 of \cite{MR2676667} tells us that for a permutation $\pi\in\mathbb{S}(A;n)$ we have $t(12,\pi)=o(1)$ asymptotically almost surely. This means that $\Delta_{(132)}$ is the antidiagonal permuton. Known bijections between $\mathbb{S}(123;n)$, $\mathbb{S}(132;n)$, $\mathbb{S}(213;n)$, $\mathbb{S}(231;n)$, $\mathbb{S}(312;n)$, and $\mathbb{S}(321;n)$ tell us that $\Delta_{(123)}=\Delta_{(213)}$ is also the antidiagonal permuton, whereas $\Delta_{(231)}=\Delta_{(312)}=\Delta_{(321)}$ is the diagonal permuton. For monotone increasing patterns it follows from a more general large deviations result~\cite{MP16} that $\Delta_{\mathsf{identity}_k}$ is the antidiagonal permuton. We do not know of other limit permutons of typical $A$-free permutations. 

All the above examples lead to the diagonal or the antidiagonal limit permuton. Therefore, it is of particular interest to find limit permutons of a different shape. For the permutation $\sigma=(3142)$, suggested to us by Mikl\'os B\'ona, $\Delta_{\sigma}$ can neither be the diagonal nor the antidiagonal (if it exists). Indeed, suppose that $\Delta_{\sigma}$ is the diagonal. Then $\Delta_{\sigma^{-1}}$ must also be the diagonal. On the other hand, we can define the complement of a permutation $\pi\in \mathbb{S}(n)$ via $C(\pi)(i)=n+1-\pi(i)$ and observe that $\Delta_{C(\sigma)}$ then must be the antidiagonal contradicting that $C(\sigma)=(2413)=\sigma^{-1}$. One can argue similarly that $\Delta_{\sigma}$ cannot be the antidiagonal.

\section{Acknowledgments}
This work was initiated during the workshop \emph{Interfaces of the Theory of Combinatorial Limits} at the Erd\H{o}s center of the Rényi Institute. We thank the organisers for creating a very productive atmosphere. We thank Mikl\'os B\'ona, Joshua Cooper, Martin Dole\v{z}al, Lara Pudwell and the two anonymous referees for comments and discussions.
\bibliographystyle{plain}
\bibliography{Relocation}

\end{document}